\documentclass[11pt,twoside,reqno,centertags,draft]{amsart}
\usepackage{amsmath,amsthm,amsfonts,amssymb}
\newtheorem{Theorem}{Theorem}

\newtheorem{theorem}[Theorem]{Theorem}
\newtheorem{lemma}[Theorem]{Lemma}

\newtheorem{remark}[Theorem]{Remark}
\newcommand{\R}{{\mathbb R}}
\newcommand{\eps}{\varepsilon}
\font\basic=cmr10
\setbox1=\hbox{\basic\strut Department of Applied Mathematics and Statistics, Comenius University} 
\setbox2=\hbox{\basic\strut Mlynsk\'a dolina, 84248 Bratislava, Slovakia}
\setbox3=\hbox{\basic\strut email: {\tt quittner@fmph.uniba.sk}}

\begin{document}

\title{An optimal Liouville theorem \\ 
for the linear heat equation \\
with a nonlinear boundary condition}
\dedicatory{Dedicated to the memory of Pavol Brunovsk\'y}
\author{Pavol Quittner \\ \\ \box1 \\ \box2 \\ \box3}
\thanks{Supported in part by the Slovak Research and Development Agency 
        under the contract No. APVV-18-0308 and by VEGA grant 1/0347/18.} 
\date{}

\begin{abstract}
Liouville theorems for scaling invariant nonlinear parabolic
problems in the whole space and/or the halfspace (saying that the problem
does not posses positive bounded solutions defined for all times
$t\in(-\infty,\infty)$)  
guarantee optimal estimates of solutions of
related initial-boundary value problems in general domains.
We prove an optimal Liouville theorem
for the linear equation in the halfspace complemented
by the nonlinear boundary condition $\partial u/\partial\nu=u^q$, $q>1$. 


\end{abstract}

\maketitle

\section{Introduction and main results}
\label{intro}

\vskip5mm
Liouville theorems for scaling invariant superlinear parabolic
problems in the whole space and/or the halfspace 
(saying that the problem
does not posses positive bounded solutions defined for all times
$t\in(-\infty,\infty)$)
guarantee optimal estimates of solutions of
related initial-boundary value problems in general domains,
including estimates of singularities and decay, see \cite{PQS}
or \cite{QS19} and the references therein.
In the case of the model problem
$$ u_t-\Delta u=u^p, \qquad x\in\R^n,\ t\in\R, $$
where $p>1$, $n\geq1$ and $u=u(x,t)>0$,
an optimal Liouville theorem 
(i.e.~a Liouville theorem valid in the full subcritical range)
has been recently proved in \cite{Q-DMJ}.
Its proof was inspired by \cite{GK} and it was based
on refined energy estimates for suitably rescaled solutions. 
In this paper we adapt the arguments in \cite{Q-DMJ}
to prove an optimal Liouville theorem for the problem
\begin{equation} \label{eq-ubc}
\left.\begin{aligned}
 u_t-\Delta u &=0 &\qquad&\hbox{in }\R^n_+\times\R, \\
 u_\nu &= u^q      &\qquad&\hbox{on }\partial\R^n_+\times\R,
\end{aligned}\quad\right\}
\end{equation}
where  $u=u(x,t)>0$,
$\R^n_+:=\{(x=(x_1,x_2,\dots,x_n)\in\R^n:x_1>0\}$, 
$\nu=(-1,0,0,\dots,0)$ is the outer unit normal on
the boundary $\partial\R^n_+=\{x\in\R^n:x_1=0\}$
and $q>1$.
In addition, we also provide an application of our Liouville theorem.

The nonexistence of positive classical stationary solutions of \eqref{eq-ubc}
is known for $q<q_S$, where
$$ q_S:=\begin{cases}
    +\infty & \hbox{ if } n\leq2, \\
    \frac{n}{n-2}  & \hbox{ if } n>2, 
\end{cases}$$
and the condition $q<q_S$ is optimal for the nonexistence,
see \cite{Hu,Har14} and the references therein.
Our main result is the following Liouville theorem.

\begin{theorem} \label{thm1}
Let $1<q<q_S$. Then problem \hbox{\rm(\ref{eq-ubc})}
does not possess positive classical bounded solutions.
\end{theorem}

The nonexistence result in Theorem~\ref{thm1}
follows from the Fujita-type results in \cite{GL,DFL} if $q\leq(n+1)/n$.
It has also been proved 
for $n=1$, $q>1$ (for solutions with bounded spatial derivatives, see \cite{QS-D}),
and for $n\geq1$ and $q<q_{sg}$ or $q=q_{sg}$ (see \cite{Q} or \cite{Q20}, respectively),
where
$$ q_{sg}:=\begin{cases}
    +\infty & \hbox{ if } n\leq2, \\
    \frac{n-1}{n-2}  & \hbox{ if } n>2. 
\end{cases}$$

Assuming on the contrary that a solution in Theorem~1 exists,
the proof of \cite[Theorem~5]{Q} guarantees that 
we may assume that the solution is nonincreasing in $x_1$.
Using this monotonicity, the proof of Theorem 1
could be obtained by more or less straightforward
modifications of the proof of \cite[Theorem 1]{Q-DMJ}.
Unfortunately, several technical arguments in Steps 4--6 
of the proof of \cite[Theorem 1]{Q-DMJ}
are written in an unnecessarily complicated way.
In order to make those arguments simpler and more transparent,
we have significantly modified the corresponding parts of the proof;
see Lemmas~\ref{lemmaG}--\ref{lem-step4} below.
Analogous modifications can also be done in the proof of \cite[Theorem~1]{Q-DMJ},
see Remark~\ref{remp} below.

Theorem~\ref{thm1} can be used in order to prove 
optimal estimates for various problems
related to \eqref{eq-ubc}.
In particular, it guarantees an optimal blow-up rate estimate
for positive solutions of the problem
\begin{equation} \label{eq-nbc}
\left.\begin{aligned}
 u_t-\Delta u &=0 &\qquad& x\in\Omega,\ t\in(0,T), \\
 u_\nu &= u^q      &\qquad& x\in\partial\Omega,\ t\in(0,T),
\end{aligned}\quad\right\}
\end{equation}
where $\Omega\subset\R^n$ is bounded and smooth.
More precisely, the proof of \cite[Theorem 7]{Q}
shows that the following theorem is true
(see the discussion after \cite[Theorem 7]{Q} for related results
and references on the blow-up rate and see also \cite{Har20}
and the references therein for results on the blow-up profile).

\begin{theorem} \label{thm3b}
Assume that $\Omega\subset\R^n$ is bounded and smooth, 
$1<q<q_S$. 
Assume also that $u$ is a positive classical solution of \hbox{\rm(\ref{eq-nbc})} 
which blows up at $t=T$. 
Then there exists $C=C(u)>0$ such that $u$
satisfies the blow-up rate estimate
$$ u(x,t)(T-t)^{1/2(q-1)}+|\nabla u(x,t)|(T-t)^{q/2(q-1)}\leq C $$
for all $x\in\overline\Omega$ and $t\in(T/2,T)$.
\end{theorem}

\section{Proof of Theorem~\ref{thm1}} 
\label{sec-proof}

The proof is a combination of arguments used in the proofs of \cite[Theorem~5]{Q} 
and \cite[Theorem~1]{Q-DMJ}.

Assume on the contrary that there exists
a positive bounded solution $u$ of (\ref{eq-ubc}).
As in the proof of \cite[Theorem~5]{Q}
we may assume 
$$ u(x,t)+|\nabla u(x,t)|\leq 1 \qquad\hbox{ for all }x\in\overline{\R^n_+},\ t\in\R, $$
and we also have
\begin{equation} \label{monotone}
 u_{x_1}(x,t)\leq0 \qquad\hbox{ for all }x\in\R^n_+,\ t\in\R.
\end{equation}

Due to the results in \cite{Q,Q20} we may also assume $q>q_{sg}$.
We set $\beta:=\frac1{2(q-1)}$
and by $C,C_0,C_1,\dots,c,c_0,c_1,\dots$ we will denote
positive constants which depend only on $n$ and $q$;
the constants $C,c$ may vary from step to step.
Finally, $M=M(n,q)$ will denote a positive integer
(the number of bootstrap steps). 
The proof will be divided into several steps.

{\bf Step 1: Initial estimates.}
For $y\in\R^n_+$, $s\in\R$, $a\in\partial\R^n_+$ and
$k=1,2,\dots$ set 
$$w(y,s)=w_k^a(y,s):=(k-t)^\beta u(y\sqrt{k-t}+a,t),\qquad\hbox{where }\  s=-\log(k-t),\ \ t<k.$$
Set also $s_k:=-\log k$ and notice that
$w=w_k^a$ solve the problem
\begin{equation} \label{eq-wbc}
\left.\begin{aligned}
 w_s &=\Delta w-\frac12 y\cdot\nabla w-\beta w 
   =\frac1\rho\nabla\cdot(\rho\nabla w)-\beta w &\quad& \hbox{in }
\R^n_+\times\R, \\
 w_\nu &= w^q &\quad& \hbox{on }\partial\R^n_+\times\R,
\end{aligned}\ \right\}
\end{equation}
where $\rho(y):=e^{-|y|^2/4}$.
In addition, $w_k^a(0,s_k)=k^\beta u(a,0)$ 
and
\begin{equation} \label{bound-w2}
 \|w_k^a(\cdot,s)\|_\infty\leq C_0k^\beta\|u(\cdot,t)\|_\infty\leq
C_0k^\beta \quad \hbox{for }\
s\in[s_k-M-1,\infty),
\end{equation}
where $t=k-e^{-s}$ and $C_0:=e^{(M+1)\beta}$.
Set
$$ 
 E(s)=E_k^a(s):=\frac12\int_{\R^n_+}\bigl(|\nabla w_k^a|^2+\beta (w_k^a)^2\bigr)(y,s)\rho(y)\,dy 
  -\frac1{q+1}\int_{\partial\R^n_+}(w_k^a)^{q+1}(y,s)\rho(y)\,dS_y.
$$
Multiplying equation (\ref{eq-wbc}) by $w\rho$ 
and integrating over $y\in\R^n_+$  we obtain
\begin{equation} \label{Ews}
E(s)=-\frac12\int_{\R^n_+}(ww_s)(y,s)\rho(y)\,dy
    +\frac12\,\frac{q-1}{q+1}\int_{\partial\R^n_+}w^{q+1}(y,s)\rho(y)\,dS_y.
\end{equation}
The function $s\mapsto E(s)$ is nonincreasing and nonnegative (see \cite{CF}).
Integrating \eqref{Ews} over the time interval $(\sigma_1,\sigma_2)$
we have
\begin{equation} \label{GK1}
\left.\begin{aligned}
\frac12&\Bigl(\int_{\R^n_+}w^2(y,\sigma_2)\rho(y)\,dy
             -\int_{\R^n_+}w^2(y,\sigma_1)\rho(y)\,dy\Bigr) \\
 &\qquad\qquad= -2\int_{\sigma_1}^{\sigma_2}E(s)\,ds
  +\frac{q-1}{q+1}\int_{\sigma_1}^{\sigma_2}\int_{\partial\R^n_+}
    w^{q+1}(y,s)\rho(y)\,dS_y\,ds.
\end{aligned}\ \right\}
\end{equation}
As in the proof of \cite[Theorem~5]{Q} we also obtain
\begin{equation} \label{GK2}
 \int_{\sigma_1}^{\sigma_2}\int_{\R^n_+} 
 \Big|\frac{\partial w}{\partial s}(y,s)\Big|^2\rho(y)\,dy\,ds
 = E(\sigma_1)-E(\sigma_2)\leq E(\sigma_1),
\end{equation}
\begin{equation} \label{monotone-int}
 \int_{\R^n_+}w^r(y,s)\rho(y)\,dy \leq C\int_{\partial\R^n_+}w^r(y,s)\rho(y)\,dS_y,
 \quad r\geq1, 
\end{equation}
\begin{equation} \label{GK4}
 \int_{\R^n_+}w(y,s)\rho(y)\,dy \leq C, 
\end{equation}
\begin{equation} \label{GK5}
 \int_{\sigma_1}^{\sigma_2}\int_{\partial\R^n_+}w^q(y,s)\rho(y)\,dS_y\,ds
 \leq C(1+\sigma_2-\sigma_1).
\end{equation} 
Given $1\leq m\leq M$,
the monotonicity  of $E$,
(\ref{GK1}), (\ref{bound-w2}), (\ref{GK4}) and (\ref{GK5})
guarantee
$$ \begin{aligned}
2&E_k^a(s_k-m) 
\leq 2\int_{s_k-m-1}^{s_k-m} E_k^a(s)\,ds \\
 &\leq \frac12\int_{\R^n_+}(w_k^a)^2(y,s_k-m-1)\rho(y)\,dy
 +\frac{q-1}{q+1}\int_{s_k-m-1}^{s_k-m}\int_{\partial\R^n_+}
   (w_k^a)^{q+1}(y,s)\rho(y)\,dS_y\,ds  \\
 &\leq Ck^\beta\Bigl( \int_{\R^n_+}w_k^a(y,s_k-m-1)\rho(y)\,dy
 +\int_{s_k-m-1}^{s_k-m}\int_{\partial\R^n_+}
   (w_k^a)^{q}(y,s)\rho(y)\,dS_y\,ds\Bigr)  \\
&\leq Ck^\beta. 
\end{aligned} $$ 
Consequently, 
\begin{equation} \label{EM}
E_k^a(s_k-M)\leq C k^\beta.
\end{equation}
Notice also that (\ref{GK2}) guarantees
\begin{equation} \label{Es}
\int_{s_k-m}^{s_k-m+1}\int_{\R^n_+} 
\Big|\frac{\partial w_k^a}{\partial s}(y,s)\Big|^2\rho(y)\,dy\,ds
\leq E_k^a(s_k-m),\quad m=1,2,\dots M.
\end{equation}

{\bf Step 2: The plan of the proof.}
We will show that there exist an integer $M=M(n,q)$ and positive numbers $\gamma_m$,
$m=1,2,\dots M$, such that
$$\gamma_1<\gamma_2<\dots<\gamma_M=\beta, \qquad \gamma_1<\mu:=2\beta-\frac{n-2}2,$$ 
and
\begin{equation} \label{E}
E_k^a(s_k-m)\leq Ck^{\gamma_m}, \qquad  a\in\partial\R^n_+,\ k\hbox{ large}, 
\end{equation}
where $m=M,M-1,\dots,1$, and ``$k$ large'' means $k\geq k_0$ with $k_0=k_0(n,q,u)$. 
Then, taking $\lambda_k:=k^{-1/2}$ and setting
$$ v_k(z,\tau):=\lambda_k^{1/(q-1)}w_k^0(\lambda_k z,\lambda_k^2\tau+s_k),
 \qquad z\in\R^n_+,\ -k\leq\tau\leq0,  $$ 
we obtain
$0<v_k\leq C$, $v_k(0,0)=u(0,0)$,
$$\begin{aligned} \frac{\partial v_k}{\partial\tau}-\Delta v_k
  &=-\lambda_k^2\Bigl(\frac12 z\cdot\nabla v_k+\beta v_k\Bigr) &\quad&\hbox{ in }\R^n_+\times(-k,0), \\ 
  (v_k)_\nu &= v_k^q &\quad&\hbox{ on }\partial\R^n_+\times(-k,0).
\end{aligned}$$
In addition,
using \eqref{Es} and \eqref{E} with $m=1$  we also have
\begin{equation} \label{estvtau}
\begin{aligned}
\int_{-k}^0\int_{|z|<\sqrt{k},\,z_1>0}
 \Big|\frac{\partial v_k}{\partial\tau}(z,\tau)\Big|^2\,dz\,d\tau
  &=\lambda_k^{2\mu}
 \int_{s_k-1}^{s_k}\int_{|y|<1,\,y_1>0} 
\Big|\frac{\partial w^0_k}{\partial s}(y,s)\Big|^2\,dy\,ds \\
&\leq C k^{-\mu+\gamma_1}\to 0 \quad\hbox{as }\ k\to\infty.
\end{aligned}
\end{equation}
Now a priori estimates of $v_k$ 
(see estimates in \cite[Theorem 7.2 and the subsequent Remark]{LSU} or \cite[Theorem 13.16]{Lieb}
applied to $v_k$ and their first order derivatives, and cf.~also
\cite[(3.9)]{HY}, for example) 
show that (up to a subsequence) the sequence $\{v_k\}$
converges to a positive solution $v=v(z)$
of the problem $\Delta v=0$ in $\R^n_+$, $v_\nu=v^q$ on $\partial\R^n_+$
which contradicts the elliptic Liouville theorem in \cite{Hu}. 
This contradiction will conclude the proof.

Notice that \eqref{E} is true if $m=M$ due to \eqref{EM}.
In the rest of the proof we consider $M>1$, fix $m\in\{M,M-1,\dots,2\}$,
assume that \eqref{E} is true with this fixed $m$, 
and we will prove that \eqref{E} remains true with $m$ replaced by $m-1$.
More precisely, we assume
\begin{equation} \label{Em}
E_k^a(s_k-m)\leq Ck^{\gamma}, \qquad  a\in\partial\R^n_+,\ k\ \hbox{large},
\end{equation}
(where $\gamma:=\gamma_m\in[\mu,\beta]$)
and we will show that
\begin{equation} \label{Em1}
  E_k^a(s_k-m+1)\leq Ck^{\tilde\gamma}, \qquad  a\in\partial\R^n_+,\ k\ \hbox{large},
\end{equation}
where $\tilde\gamma<\gamma$ (and then we set $\gamma_{m-1}:=\tilde\gamma$).
Our proof shows that there exists an open neighbourhood $U=U(n,q,\gamma)$ of $\gamma$
such that
\eqref{Em1} remains true also if \eqref{Em}
is satisfied with $\gamma$ replaced by any $\gamma'\in U$.
The compactness of $[\mu,\beta]$ guarantees that 
the difference $\gamma-\tilde\gamma$ can be bounded below by a positive constant
$\delta=\delta(n,q)$ for all $\gamma\in[\mu,\beta]$, hence
there exists $M=M(n,q)$ such that $\gamma_1<\mu\leq\gamma_2$.  

{\bf Step 3: Notation and auxiliary results.}
In the rest of the proof we will also use the following notation
and facts:
If $Z$ is a finite set or a measurable subset of $\R^d$, then by $\#Z$ or $|Z|$ we denote 
the cardinality or the $d$-dimensional measure of $Z$, respectively.
Set 
$$\begin{aligned}
 C(M)&:=8ne^{M+1}, &\ \ B^\partial_r(a)&:=\{x\in\partial\R^n_+:|x-a|\leq r\}, &\ \ B^\partial_r&:=B^\partial_r(0), \\
 R_k&:=\sqrt{8n\log k}, &\ \ B^+_r(a)&:=\{x\in\R^n_+:|x-a|\leq r\}, &\ \ B^+_r&:=B^+_r(0). 
\end{aligned} $$
Given $a\in\partial\R^n_+$, there exists an integer $X=X(n,k)$ and
there exist $a^1,a^2,\dots a^X\in\partial\R^n_+$ (depending on $a,n,k$) such that
$a^1=a$, $X\leq C(\log k)^{(n-1)/2}$ and
\begin{equation} \label{B}
 D^k(a):=B^\partial_{\sqrt{C(M)k\log(k)}}(a)\subset\bigcup_{i=1}^X B^\partial_{\sqrt{k}/2}(a^i).
\end{equation}
Notice that if $y\in B^\partial_{R_k}$ and $s\in[s_k-M-1,s_k]$,
then $a+ye^{-s/2}\in D^k(a)$, hence
\eqref{B} guarantees the existence of $i\in\{1,2,\dots,X\}$
such that 
\begin{equation} \label{yyi}
w^a_k(y,s)=w^{a^i}_k(y^i,s),\qquad\hbox{where}\quad
 y^i:=y+(a-a^i)e^{s/2}\in B^\partial_{1/2}.
\end{equation}

The contradiction argument in Step~2 based on 
the nonexistence of positive stationary solutions of \eqref{eq-ubc},
combined with a doubling argument
can be used to obtain the following
useful pointwise estimates of the solution $u$.

\begin{lemma} \label{lem-decay}
Let $M,s_k,w^a_k$ be as above,
$\zeta\in\R$, $\xi,C^*>0$, $d_k,r_k\in(0,1]$, $k=1,2,\dots$.
Set
$$ \begin{aligned}
{\mathcal T}_k &={\mathcal T}_k(d_k,r_k,\zeta,C^*) \\
 &:=\Bigl\{(a,\sigma,b)\in\partial\R^n_+\times(s_k-M,s_k]\times\partial\R^n_+:
  \int_{\sigma-d_k}^{\sigma}\int_{B^+_{r_k}(b)} (w^a_k)_s^2 dy\,ds \leq C^*k^\zeta
\Bigr\}. \end{aligned}$$ 
Assume
\begin{equation} \label{xizeta}
 \xi\frac\mu\beta>\zeta\quad{and}\quad
 \frac1{\log(k)}\min(d_kk^{\xi/\beta},r_kk^{\xi/2\beta})\to\infty\ 
 \hbox{ as }\ k\to\infty.
\end{equation}
Then there exists $k_1$ such that
$$ w^a_k(y,\sigma)\leq k^\xi \quad\hbox{whenever}\quad
 y\in B^\partial_{r_k/2}(b), \ \  k\geq k_1 \ 
\hbox{ and } \ (a,\sigma,b)\in{\mathcal T}_k.$$ 
\end{lemma}

\begin{proof}
Assume on the contrary that there exist $k_1,k_2\dots$ with the following properties:
$k_j\to\infty$ as $j\to\infty$, and for each $k\in\{k_1,k_2,\dots\}$ there exist
$(a_k,\sigma_k,b_k)\in{\mathcal T}_k$ and $y_k\in B^\partial_{r_k/2}(b_k)$ such that
$\tilde w_k(y_k,\sigma_k)>k^\xi$, where $\tilde w_k:=w^{a_k}_{k}$.

Given $k\in\{k_1,k_2\,\dots\}$, we can choose an integer $K$ such that
\begin{equation} \label{Klemma}
  2^Kk^{\xi}>C_0k^\beta, \qquad K<C\log k.
\end{equation}
Set
$$ Z_j:=B^\partial_{r_k(1/2+j/(2K))}(b_k)\times[\sigma_k-d_k(1/2+j/(2K)),\sigma_k],
 \quad j=0,1,\dots,K.$$
Then
$$ B^\partial_{r_k/2}(b_k)\times[\sigma_k-d_k/2,\sigma_k]=Z_0\subset Z_1\subset\dots\subset Z_K
  =B^\partial_{r_k}(b_k)\times[\sigma_k-d_k,\sigma_k].$$
Since $\sup_{Z_0}\tilde w_k\geq \tilde w_k(y_k,\sigma_k)>k^{\xi}$,
estimates \eqref{Klemma} and \eqref{bound-w2} imply the existence of $j^*\in\{0,1,\dots K-1\}$ such that
$$ 2\sup_{Z_{j^*}}\tilde w_k \geq \sup_{Z_{j^*+1}}\tilde w_k $$
(otherwise $C_0k^\beta\geq\sup_{Z_K}\tilde w_k>2^K\sup_{Z_0}\tilde w_k>2^K k^\xi$, a contradiction).
Fix $(\hat y_k,\hat s_k)\in Z_{j^*}$ such that
$$ W_k:=\tilde w_k(\hat y_k,\hat s_k)=\sup_{Z_{j^*}}\tilde w_k.$$
Then $W_k\geq  k^{\xi}$,
$B^\partial_{r_k/(2K)}(\hat y_k)\times\Bigl[\hat s_k-\frac{d_k}{2K},\hat s_k\Bigr]\subset Z_{j^*+1}$,
hence \eqref{monotone} implies
$$\tilde w_k\leq 2W_k \ \hbox{ on } \
 \hat Q_k:=B^+_{r_k/(2K)}(\hat y_k)\times\Bigl[\hat s_k-\frac{d_k}{2K},\hat s_k\Bigr]. $$

Set $\lambda_k:=W_k^{-1/(2\beta)}$ (hence $\lambda_k\leq k^{-\xi/(2\beta)}\to 0$ as $k\to\infty$)
and
$$ v_k(z,\tau):=\lambda_k^{2\beta}\tilde w_k(\lambda_k z+\hat y_k,\lambda_k^2\tau+\hat s_k).$$
Then $v_k(0,0)=1$, $v_k\leq2$ on $Q_k:=B^+_{r_k/(2K\lambda_k)}\times[-d_k/(2K\lambda_k^2),0]$, and
\begin{equation} \label{eqvk} 
 \begin{aligned}
  \frac{\partial v_k}{\partial\tau}-\Delta v_k
  &=-\lambda_k^2\Bigl(\frac12 z\cdot\nabla v_k+\beta v_k\Bigr) &\quad&\hbox{in }\ Q_k, \\
  (v_k)_\nu &= v_k^q &\quad&\hbox{on }\ Q^\partial_k, 
\end{aligned}
\end{equation}
where $Q^\partial_k:=B^\partial_{r_k/(2K\lambda_k)}\times[-d_k/(2K\lambda_k^2),0]$.
In addition, as $k\to\infty$,
$$
  \frac{r_k}{2K\lambda_k} \geq\frac{r_k k^{\xi/(2\beta)}}{C\log(k)}\to\infty, \quad
  \frac{d_k}{2K\lambda_k^2} \geq\frac{d_k k^{\xi/\beta}}{C\log(k)}\to\infty.
$$
Since $(a_k,\sigma_k,b_k)\in{\mathcal T}_k$ and 
$\hat Q_k\subset B^+_{r_k}(b_k)\times[\sigma_k-d_k,\sigma_k]$, we obtain
$$
\int_{Q_k}\Big|\frac{\partial v_k}{\partial\tau}(z,\tau)\Big|^2\,dz\,d\tau
  =\lambda_k^{2\mu}\int_{\hat Q_k}
\Big|\frac{\partial \tilde w_k}{\partial s}(y,s)\Big|^2\,dy\,ds \leq C^*k^\delta,
\quad\hbox{where }\ \delta:=-\xi\frac\mu\beta+\zeta <0.
$$
Hence, as above, a suitable subsequence of $\{v_k\}$
converges to a positive solution $v=v(z)$
of the problem $\Delta v=0$ in $\R^n_+$, $v_\nu=v^q$ on $\partial\R^n_+$, 
which contradicts the elliptic Liouville theorem in \cite{Hu}.
 \end{proof}

\begin{remark} \rm
By a simple modification of the proof of Lemma~\ref{lem-decay}
one can show that the estimate 
$w^a_k(y,\sigma)\leq k^\xi$  can be improved to
$$ w^a_k(y,\sigma)+|\nabla w^a_k(y,\sigma)|^{1/q}+|(w^a_k)_s(y,\sigma)|^{1/(2q-1)}
 \leq k^\xi.$$
In fact, set $\tilde w_k:=w^{a_k}_k+|\nabla w^{a_k}_k|^{1/q}+|(w^{a_k}_k)_s|^{1/(2q-1)}$ 
and assume on the contrary that
$\tilde w_k(y_k,\sigma_k)>k^\xi$, where $y_k,\sigma_k,k$ are as in the proof of Lemma~\ref{lem-decay}.
Repeat the doubling estimates (with a modified constant $C_0$)
and define $W_k$ and $\lambda_k$ as in that proof,
but replace $\tilde w_k$ with $w^{a_k}_k$ in the definition of $v_k$.
Then $v_k$ solves \eqref{eqvk},
$$ \begin{aligned}
\bigl(v_k+|\nabla v_k|^{1/q}+|(v_k)_\tau|^{1/(2q-1)}\bigr)(0,0)&=1, \\ 
v_k+|\nabla v_k|^{1/q}+|(v_k)_\tau|^{1/(2q-1)}&\leq2\ \hbox{ in $Q_k$},
\end{aligned}$$ 
and passing to the limit we arrive at a contradiction.
\qed
\end{remark}

Recall that $\gamma\in[\mu,\beta]$ (see \eqref{Em}).

\begin{lemma} \label{lemmaG}
Let ${\mathcal T}_k={\mathcal T}_k(d_k,r_k,\zeta,C^*)$ be as in Lemma~\ref{lem-decay},
$\omega\in\R$, $\eps,C_1>0$,
$$ 0\leq\alpha<\frac\xi\beta,\qquad \xi\frac\mu\beta>\gamma-\alpha+\eps-\omega, $$ 
and assume 
\begin{equation} \label{assG}
 \hbox{$(a,\sigma,0)\in {\mathcal T}_k(\frac12 k^{-\alpha},1,\gamma-\alpha+\eps,C_1)$\ \ for $k$ large}.
\end{equation}
Set 
$$G:=\{y\in B^\partial_{1/2}: w^a_k(y,\sigma)\leq k^\xi\}.$$
Then 
\begin{equation} \label{Gc}
|B^\partial_{1/2}\setminus G|\leq Ck^{\omega-(n-1)\alpha/2}\ \hbox{ for $k$ large}.
\end{equation}
\end{lemma}

\begin{proof}
There exist 
$b^1,\dots,b^{Y}\in\partial\R^n_+$ with $Y\leq Ck^{(n-1)\alpha/2}$
such that 
$$ B^\partial_{1/2}\subset \bigcup_{j=1}^{Y}B^j,\quad\hbox{where}\quad
  B^j:=B^\partial_{\frac12k^{-\alpha/2}}(b^j), $$
and 
\begin{equation} \label{multiynew}
 \#\{j: y\in B^+_{k^{-\alpha/2}}(b^j)\}\leq C_n\quad\hbox{for any }\ y\in\R^n_+.
\end{equation}
Set
$$ \begin{aligned}
  H &:=\bigl\{j\in\{1,2,\dots,Y\}: 
  (a,\sigma,b^j)\in{\mathcal T}_k(\hbox{$\frac12$}k^{-\alpha},k^{-\alpha/2},\gamma-\alpha+\eps-\omega,C_1C_n)\bigr\}, \\
  H^c &:=\{1,2,\dots,Y\}\setminus H.
\end{aligned}
$$
If $j\in H$, then  Lemma~\ref{lem-decay} 
guarantees $w^a_k(y,\sigma)\leq k^\xi$ for $y\in B^j$.
Consequently, 
$$B^\partial_{1/2}\cap\bigcup_{j\in H}B^j \subset G, \ \hbox{ hence } \
  B^\partial_{1/2}\setminus G \subset \bigcup_{j\in H^c}B^j. $$
Now \eqref{assG}, the definition of $H$ and \eqref{multiynew} imply
$\#H^c< k^\omega$, 
hence \eqref{Gc} is true.
\end{proof}

\begin{lemma} \label{lem-ineq}
Fix a positive integer $L=L(n,q)$ such that
\begin{equation} \label{estL}
 \beta(\frac{q+1}{3q-1})^L<\mu. 
\end{equation}
If $\eps,\delta>0$ are small enough, then there exist
$\xi_\ell,\alpha_\ell,\omega_\ell$, $\ell=1,2,\dots L$, 
such that 
\begin{equation} \label{estxi}
\gamma-\delta-\eps >\xi_1\leq\xi_2\leq\dots\leq\xi_L\leq\beta=:\xi_{L+1}
\end{equation}
and the following inequalities are true for $\ell=1,2,\dots,L$:
\begin{equation} \label{bootcond}
0\leq\alpha_\ell <\frac{\xi_\ell}\beta, \quad \xi_l\frac\mu\beta>\gamma-\alpha_\ell+\eps-\omega_\ell, \quad
\omega_\ell-\frac{(n-1)\alpha_\ell}2\leq\gamma-\delta-(q+1)\xi_{\ell+1}.
\end{equation}
\end{lemma}

\begin{proof}
Consider $\xi\in[\mu/2,\beta]$ and 
$\tilde\xi\in[\xi,\frac{3q-1}{q+1}\xi)$.
Set also $\alpha:=\frac\xi\beta-\eps_\alpha$, where $\eps_\alpha>0$ is small.
Since
$$ \gamma-\frac\xi\beta-\xi\frac\mu\beta =\frac{n-1}2\frac\xi\beta+\gamma-(3q-1)\xi$$
and $(q+1)\tilde\xi<(3q-1)\xi$, we see that
\begin{equation} \label{nu} 
  \underline\omega:= \gamma-\alpha+\eps-\xi\frac\mu\beta<
  \frac{n-1}2\alpha+\gamma-\delta-(q+1)\tilde\xi=:\overline\omega
\end{equation}
provided $\eps,\eps_\alpha,\delta$ are small enough.
Consequently, we may choose $\omega\in(\underline\omega,\overline\omega)$.

If $\eps,\delta$ are small  enough, 
then \eqref{estL} guarantees the existence of
$\xi_1,\dots,\xi_L$ satisfying \eqref{estxi}, $\xi_1\geq\mu/2$ 
and $\xi_{\ell+1}<\frac{3q-1}{q+1}\xi_\ell$ for $\ell=1,2,\dots,L$. 
Fix $\ell\in\{1,2,\dots,L\}$, set $\xi:=\xi_\ell$, $\tilde\xi:=\xi_{\ell+1}$,
and let $\alpha,\omega$ be as above. Set  $\alpha_\ell:=\alpha$, $\omega_\ell:=\omega$.
Then the definitions of $\alpha_\ell,\omega_\ell$ and \eqref{nu} guarantee \eqref{bootcond}.
\end{proof}

\begin{lemma} \label{lem-wq}
Let $L,\eps,\delta$ and 
$\xi_\ell,\alpha_\ell,\omega_\ell$, $\ell=1,2,\dots L$, be as in
Lemma~\ref{lem-ineq} and let
${\mathcal T}_k={\mathcal T}_k(d_k,r_k,\zeta,C^*)$ be as in Lemma~\ref{lem-decay}.
Assume $(a,\sigma,0)\in {\mathcal T}_k(\frac12 k^{-\alpha_\ell},1,\gamma-\alpha_\ell+\eps,C)$
for $\ell=1,2,\dots L$ and $k$ large and
\begin{equation} \label{estKaplan} 
  \int_{B^\partial_{1/2}}(w^a_k)^{q}(y,\sigma)\,dS_y\leq Ck^{\eps}\ \hbox{ for $k$ large}.
\end{equation}
Then 
\begin{equation} \label{est-wq}
\int_{B^\partial_{1/2}}(w^a_k)^{q+1}(y,\sigma)\,dS_y\leq Ck^{\gamma-\delta}\ \hbox{ for $k$ large}.
\end{equation}
\end{lemma}

\begin{proof}
Given $\ell\in\{1,2,\dots,L\}$, set $\xi=\xi_\ell$, $\alpha=\alpha_\ell$, $\omega=\omega_\ell$,
and let $G$ be the set in Lemma~\ref{lemmaG}.
Set $G_\ell:=G$ and $G_{L+1}:=B^\partial_{1/2}$. 
Lemma~\ref{lemmaG} and \eqref{bootcond} guarantee
$$|G_{\ell+1}\setminus G_\ell|\leq|B^\partial_{1/2}\setminus G_\ell|\leq Ck^{\omega_\ell-(n-1)\alpha_\ell/2}
                 \leq Ck^{\gamma-\delta-(q+1)\xi_{\ell+1}},$$
hence
$$\int_{G_{\ell+1}\setminus G_\ell}(w^a_k)^{q+1}(y,\sigma)\,dS_y
 \leq Ck^{(q+1)\xi_{\ell+1}}|G_{\ell+1}\setminus G_\ell| \leq Ck^{\gamma-\delta}.$$
In addition, the definition of $G_1$, \eqref{estKaplan} and \eqref{estxi} imply
\begin{equation} \label{estG1}
 \int_{G_1}(w^a_k)^{q+1}(y,\sigma)\,dS_y \leq k^{\xi_1}\int_{G_1}(w^a_k)^{q}(y,\sigma)\,dS_y 
 \leq Ck^{\xi_1+\eps} \leq Ck^{\gamma-\delta}.
\end{equation}
Since $B^\partial_{1/2}=G_1\cup\bigcup_{\ell=1}^L (G_{\ell+1}\setminus G_\ell)$,
the conclusion follows.
\end{proof}

{\bf Step 4: The choice of a suitable time.}
The proof of \eqref{Em1} will be based on estimates of $w^{a^i}_k(\cdot,s^*)$,
$i=1,2,\dots,X$, where
$s^*=s^*(k,a)\in[s_k-m,s_k-m+1]$ is a suitable time. 

\begin{lemma} \label{lem-step4}  
Let $\eps,\gamma,C_1,C_2>0$, $\alpha_1,\alpha_2,\dots,\alpha_L\geq0$,
and, given $k=1,2,\dots$, let $X_k$ be a positive integer satisfying 
$X_k\leq k^{\eps/2}$ and $\sigma_k\in\R$. Set $J_k:=[\sigma_k,\sigma_k+1]$, 
$\tilde J_k:=[\sigma_k+1/2,\sigma_k+1]$, and
assume that  $f^1_k,\dots,f^{X_k}_k,g^1_k,\dots,g^{X_k}_k\in C(J_k,\R^+)$
satisfy 
\begin{equation} \label{figi}
\int_{J_k}f^i_k(s)\,ds\leq C_1k^\gamma, \quad \int_{J_k}g^i_k(s)\,ds\leq C_2, \quad
i=1,2,\dots X_k,\ k=1,2,\dots.
\end{equation}
Then there exists $k_1=k_1(\eps,L)$ with the following property:
If $k\geq k_1$, then there exists $s^*=s^*(k)\in\tilde J_k$ such that 
$$ \int_{s^*-\frac12k^{-\alpha_\ell}}^{s^*}f^i_k(s)\,ds\leq C_1k^{\gamma-\alpha_\ell+\eps},
\quad  f^i_k(s^*)\leq C_1k^{\gamma+\eps}, \quad g^i_k(s^*)\leq C_2k^\eps $$
for all $i=1,2,\dots,X_k$ and $\ell=1,2,\dots,L$.
\end{lemma}

\begin{proof}
Set
$$h^{i,\ell}_k(s):=\int_{s-\frac12k^{-\alpha_\ell}}^{s}f^i_k(\tau)\,d\tau, \ \ s\in\tilde J_k,\ \
i=1,2,\dots X_k,\ \ \ell=1,2,\dots,L,\ \ k=1,2,\dots.$$
Then 
\begin{equation} \label{hiell}
\begin{aligned}
\int_{\tilde J_k} &h^{i,\ell}_k(s)\,ds
  = \int_{\tilde J_k}\int_{s-\frac12k^{-\alpha_\ell}}^{s}f^i_k(\tau)\,d\tau\,ds
  = \int_{\tilde J_k}\int_0^{\frac12k^{-\alpha_\ell}}f^i_k(s-\tau)\,d\tau\,ds \\
 &= \int_0^{\frac12k^{-\alpha_\ell}}\int_{\tilde J_k}f^i_k(s-\tau)\,ds\,d\tau
  \leq \int_0^{\frac12k^{-\alpha_\ell}}\int_{J_k} f^i_k(s)\,ds\,d\tau\leq C_1k^{\gamma-\alpha_\ell}.
\end{aligned}
\end{equation}
Set
$$ \begin{aligned}
A^i_k &:=\{s\in\tilde J_k: f^i_k(s)>C_1k^{\gamma+\eps}\},\quad
   B^i_k :=\{s\in\tilde J_k: g^i_k(s)>C_2k^{\eps}\},\\
   C^{i,\ell}_k &:=\{s\in\tilde J_k: h^{i,\ell}_k(s)>C_1k^{\gamma-\alpha_\ell+\eps}\}.
\end{aligned}$$
Then \eqref{figi} and \eqref{hiell} show that the measure of each of the sets
$A^i_k,B^i_k,C^{i,\ell}_k$ can be estimated above by $k^{-\eps}$.
Since the number of these sets (with given index $k$) is $(L+2)X_k\leq (L+2)k^{\eps/2}$,
their union $U_k:=\bigcup_i A^i_k\cup\bigcup_i B^i_k\cup\bigcup_{i,\ell}C^{i,\ell}_k$
has measure less than 1/2 for $k\geq k_1$, hence for $k\geq k_1$ 
there exists $s^*=s^*(k)\in\tilde J_k\setminus U_k$.
Obviously, $s^*$ has the required properties.
\end{proof}

Consider $m$, $\gamma\in[\mu,\beta]$ and $a\in\partial\R^n_+$ fixed,
$J_k:=[s_k-m,s_k-m+1]$,
and let $a^i$, $i=1,2,\dots,X$ be as in \eqref{B} (recall that $a^i$ and $X$ depend on $k$;
$X\leq C(\log k)^{(n-1)/2}$).
Let $L,\eps$ and $\alpha_\ell$, $\ell=1,2,\dots,L$ be from Lemma~\ref{lem-ineq}.
Set 
$$f^i_k(s):=\int_{\R^n_+}(w^{a^i}_k)_s^2(y,s)\rho(y)\,dy,\quad
g^i_k(s):=\int_{\partial\R^n_+}(w^{a^i}_k)^q(y,s)\rho(y)\,dS_y, \quad i=1,2,\dots X.$$
Estimates \eqref{Es}, \eqref{Em} and \eqref{GK5} guarantee
that the assumptions of  Lemma~\ref{lem-step4} are satisfied
with $C_1,C_2$ independent of $a$.
Consequently, if $k\geq k_1$, then
there exists $s^*=s^*(k,a)\in\tilde J_k:=[s_k-m+1/2,s_k-m+1]$ such that the following estimates are true  
for $a\in\partial\R^n$, $w=w_k^{a^i}$,
$i=1,2,\dots X$, $\ell=1,2,\dots L$: 
\begin{equation} \label{star1}
 \int_{s^*-\frac12k^{-\alpha_\ell}}^{s^*}\int_{\R^n_+}w_s^2\rho\,dy\,ds \leq C_1k^{\gamma-\alpha_\ell+\eps}, 
\end{equation}
\begin{equation} \label{star2}
\left. \begin{aligned}
 \int_{\R^n_+}w_s^2(y,s^*)\rho(y)\,dy &\leq C_1k^{\gamma+\eps}, \\
 \int_{\partial\R^n_+}w^q(y,s^*)\rho(y)\,dS_y &\leq C_2k^\eps.
\end{aligned}\quad\right\} 
\end{equation}

{\bf Step 5: Energy estimates.}
Let $L,\eps,\delta$ and 
$\xi_\ell,\alpha_\ell,\omega_\ell$, $\ell=1,2,\dots L$, be as in
Lemma~\ref{lem-ineq} and let
${\mathcal T}_k={\mathcal T}_k(d_k,r_k,\zeta,C^*)$ be as in Lemma~\ref{lem-decay}.
Let $a\in\partial\R^n_+$ be fixed and $s^*=s^*(k,a)$ be from Step~4.
Notice that \eqref{star1} guarantees 
$(a^i,s^*,0)\in{\mathcal T}_k(\frac12k^{-\alpha_\ell},1,\gamma-\alpha_\ell+\eps,C_1/\rho(1))$
for $i=1,2,\dots,X$ and $\ell=1,2,\dots,L$,
and \eqref{star2} implies \eqref{estKaplan} with $a=a^i$, $i=1,2,\dots,X$.
Consequently, Lemma~\ref{lem-wq} implies
$$  \int_{B^\partial_{1/2}}(w^{a^i}_k)^{q+1}(y,s^*)\,dS_y\leq Ck^{\gamma-\delta}\ 
 \hbox{ for $i=1,2,\dots,X$ and $k$ large}, $$ 
and using \eqref{yyi} we obtain
$$ \begin{aligned}
\int_{B^\partial_{R_k}}(w^a_k)^{q+1}(y,s^*)\,dS_y
 &\leq \sum_{i=1}^X\int_{B^\partial_{1/2}}(w^{a^i}_k)^{q+1}(y,s^*)\,dS_y
 \leq Ck^{\gamma-\delta}(\log k)^{(n-1)/2} \\
 &\leq Ck^{\gamma-\delta/2}\quad  \hbox{ for $k$ large}.
\end{aligned} $$ 
In addition, since
$$  \rho(y)=e^{-|y|^2/8-|y|^2/8}\leq k^{-n}e^{-|y|^2/8},\quad\hbox{for }\ |y|>R_k,$$
we have
$$
\int_{\partial\R^n_+\setminus B^\partial_{R_k}}(w^a_k)^{q+1}(y,s^*)\,dS_y
 \leq C\int_{\partial\R^n_+\setminus B^\partial_{R_k}}k^{(q+1)\beta-n}e^{-|y|^2/8}\,dS_y  \leq C, $$
hence
\begin{equation} \label{estwq1}
\int_{\partial\R^n_+}(w^a_k)^{q+1}(y,s^*)\,dS_y \leq Ck^{\gamma-\delta/2}\quad  \hbox{ for $k$ large}.
\end{equation}
Denoting $w:=w^a_k$, \eqref{star2}, \eqref{monotone-int} and \eqref{estwq1} imply
\begin{equation} \label{intwwsnew}
\begin{aligned}
\Big|\int_{\R^n_+} &(ww_s)(y,s^*)\rho(y)\,dy\Big| 
 \leq \Bigl(\int_{\R^n_+} w^2(y,s^*)\rho(y)\,dy\Bigr)^{1/2}
       \Bigl(\int_{\R^n_+} w_s^2(y,s^*)\rho(y)\,dy\Bigr)^{1/2} \\
  &\leq C\Bigl(\int_{\R^n_+} w^{q+1}(y,s^*)\rho(y)\,dy\Bigr)^{1/(q+1)}
         k^\frac{\gamma+\eps}2 
   \leq Ck^{\frac{\gamma-\delta/2}{q+1}+\frac{\gamma+\eps}2}
   \leq Ck^{\gamma-\delta/2}, 
\end{aligned} 
\end{equation}
provided $\eps$ and $\delta$ are small enough.
Finally, \eqref{Ews}, \eqref{estwq1} and \eqref{intwwsnew} guarentee
$E_k^a(s^*)\leq  Ck^{\tilde\gamma}$
with $\tilde\gamma:=\gamma-\delta/2$ and $k$ large,
and the monotonicity of $E_k^a$ implies \eqref{Em1}.
This concludes the proof.
\qed

\begin{remark}  \label{remp}
\rm
A straighforward modification of the proof of Theorem~1 provides
a simpler proof of \cite[Theorem~1]{Q-DMJ}. One just has
to replace $q$ with $p$, set $\beta:=1/(p-1)$,
replace $\R^n_+,\partial\R^n_+$ and $B^+_r,B^\partial_r$
with $\R^n$ and $B_r$, respectively, $(n-1)$ with $n$, and do a few more
straightforward changes. In particular, \eqref{monotone} should be removed and
\eqref{estL} should be replaced with $\beta\bigl(\frac{p+1}{2p}\bigr)^L<\mu$.
\end{remark}

\begin{remark}  \label{rempstar}
\rm
If $\beta<\mu+1$, then Lemmas~\ref{lemmaG}--\ref{lem-ineq} are not needed
and the proof of (a modification of) Lemma~\ref{lem-wq} is simpler. In fact,
the inequality $\beta<\mu+1$ implies 
$\gamma\frac\mu\beta>\gamma-\frac\gamma\beta$, hence if $\delta,\eps>0$
are small enough and $\xi:=\gamma-\eps-\delta$, then there exists
$\alpha<\xi/\beta$ such that $\xi\frac\mu\beta>\gamma-\alpha+\eps$.
If $(a,\sigma,0)\in{\mathcal T}_k(\frac12k^{-\alpha},1,\gamma-\alpha+\eps,C)$, then
Lemma~\ref{lem-decay} guarantees $w^a_k(\cdot,\sigma)\leq k^\xi$ on $B^\partial_{1/2}$,
hence assumption \eqref{estKaplan} implies
\eqref{estG1} with $G_1:=B^\partial_{1/2}$ and $\xi_1:=\xi$.
Consequently, \eqref{est-wq} is true.

Notice that if $\beta=1/(p-1)$, then the inequality $\beta<\mu+1$ is equivalent
to the inequality $p<p^*$ in \cite{Q-DMJ}.
\end{remark}


\end{document}